\documentclass[12pt]{amsart}

\usepackage{amscd}
\usepackage{amsmath}
\usepackage{amssymb}
\usepackage{amsthm}
\usepackage{epsf}
\usepackage{verbatim}
\usepackage{mathrsfs}
\input epsf.tex

\usepackage{tikz}
\usepackage{caption}
\usetikzlibrary{decorations.pathreplacing}
\usetikzlibrary{shapes.geometric}

\theoremstyle{definition}
\newtheorem{definition}{Definition}[section]

\theoremstyle{plain}
\newtheorem{theorem}[definition]{Theorem}
\newtheorem{lemma}[definition]{Lemma}
\newtheorem{corollary}[definition]{Corollary}
\newtheorem{proposition}[definition]{Proposition}

\theoremstyle{definition}
\newtheorem{remark}[definition]{Remark}

\newcommand{\R}{\mathbb{R}}

\newcommand{\Z}{\mathbb{Z}}
\newcommand{\N}{\mathbb{N}}
\newcommand{\ve}{\varepsilon}

\newcommand{\B}{\mathscr{B}}
\newcommand{\T}{\mathcal{T}}
\newcommand{\ES}{\mathcal{S}}
\newcommand{\TInv}{T^{-1}}
\newcommand{\bfv}{\textnormal{\textbf{v}}}
\newcommand{\bfu}{\textbf{u}}

\newcommand{\SQ}{\mathcal{SQ}}
\newcommand{\Loc}{\text{Loc}}

\newcommand{\D}{\mathfrak{D}}
\newcommand{\C}{\mathfrak{C}}

\begin{document}
\title[On conservative sequences and ergodic multipliers]{On conservative sequences and their application to ergodic multiplier problems}
\date{\today}
\author[Elyze]{Madeleine Elyze}
\address[Madeleine Elyze]{Williams College, Williamstown, MA 01267, USA}
\email{mag2@williams.edu}

\author[Kastner]{Alexander Kastner}
\address[Alexander Kastner]{Williams College, Williamstown, MA 01267, USA}
\email{ask2@williams.edu}

\author[Ortiz Rhoton]{Juan Ortiz Rhoton}
\address[Juan Ortiz Rhoton]{MIT, Cambridge, MA 02139, USA}
\email{jor@mit.edu}

\author[Semenov]{Vadim Semenov }
\address[Vadim Semenov]{Courant Institute, New York University, New York, NY 10012, USA }
\email {vadim.semenov@nyu.edu}

\author[Silva]{Cesar E. Silva}
\address[Cesar E. Silva]{Department of Mathematics\\
     Williams College \\ Williamstown, MA 01267, USA}
\email{csilva@williams.edu}


\maketitle

\begin{abstract}
The \emph{conservative sequence} of a set $A$ under a transformation $T$ is  the set of all $n \in \Z$ such that $T^n A \cap A \not = \varnothing$. By studying these sequences, we prove that given any countable collection of nonsingular transformations with no finite invariant measure $\{T_i\}$, there exists a rank-one transformation $S$ such that $T_i \times S$ is not ergodic for all $i$. Moreover, $S$ can be chosen to be rigid or have infinite ergodic index. We establish similar results for $\Z^d$ actions and flows. Then, we find sufficient conditions on rank-one transformations $T$ that guarantee the existence of a rank-one transformation $S$ such that $T \times S$ is ergodic, or, alternatively, conditions that guarantee that $T \times S$ is conservative but not ergodic. In particular, the infinite Chac\'on transformation satisfies both conditions. Finally, for a given ergodic transformation $T$, we study the Baire categories of the sets $E(T)$, $\bar{E}C(T)$ and $\bar{C}(T)$ of transformations $S$ such that $T \times S$ is ergodic, ergodic but not conservative, and conservative, respectively.
\end{abstract}

\section{Introduction}

In \cite{FuWe78}, Furstenberg and Weiss proved that a finite measure-preserving transformation $T$ satisfies the mild mixing property if and only if  for all finite or infinite measure-preserving ergodic transformations $S$, the cartesian product $T \times S$ is ergodic. In this case we say that mild mixing is an
{\it ergodic multiplier property}. The situation in infinite measure is quite different. It was shown in \cite{Br70}, \cite{AaLiWe79} that if $T$ is an infinite measure-preserving ergodic transformation, or a nonsingular ergodic transformation with no equivalent finite invariant measure, then there always exists an ergodic Markov shift $S$ so that the product $T \times S$ is not conservative, hence not ergodic. Thus, there is no ergodic multiplier property for infinite measure-preserving transformations.

The \emph{conservative sequence} of a set $A$ under a transformation $T$ is defined as $C_T(A) := \{n \in \Z: \mu(T^n A \cap A) > 0\}$. In this paper, we analyze the combinatorics of these sequences and use them to give an alternative proof of the aforementioned result of \cite{Br70}, \cite{AaLiWe79}. In addition, our methods allow us to extend the result in the following way. Given any countable collection of nonsingular transformation $\{T_n\}$ with no finite invariant measure, we construct a single rank-one transformation $S$ such that $T_n \times S$ is not conservative, hence not ergodic, for each $n$. Moreover, $S$ can be chosen to be rigid or to have infinite ergodic index. We also study related questions for infinite measure-preserving $\Z^d$-actions and flows. In this context we note that Schmidt and Walters \cite{ScWa82} showed a similar result to the one in \cite{Br70}, \cite{AaLiWe79} for nonsingular actions of locally compact abelian groups, but again we are interested in the multiplier being rank-one, and our methods construct the rank-one action using conservative sets. Section 2 covers the preliminary definitions and Section 3 has our main results.

In Section 4, we employ conservative sequences to establish a sufficient condition on rank-one transformations $T$ that guarantees the existence of a rank-one transformation $S$ such that $T \times S$ is conservative but not ergodic. 
This condition will be general enough to include many infinite rank-one transformations of interest such as the infinite Chac\'on and Hajian-Kakutani transformations. We also discuss in this section that one cannot hope that for all infinite measure-preserving ergodic $T$ there is an ergodic $S$ such that $T\times S$ is conservative but not ergodic.

In Section 5, we show that if $T$ is a  rank-one transformation  with bounded cuts, then there is another rank-one transformation $S$ such that $T \times S$ is ergodic. In Section 6, for a given ergodic transformation $T$, we study the Baire categories of the sets $E(T)$, $\bar{E}C(T)$ and $\bar{C}(T)$ of transformations $S$ such that $T \times S$ is ergodic, ergodic but not conservative, and conservative, respectively. In particular, we show that for a generic transformation $T$, the set $E(T)$ is a generic subset of the measure-preserving transformations.

\textbf{Acknowledgments:} This paper is based on research by the ergodic theory group of the 2016 SMALL undergraduate research project at Williams College. Support for the project was provided by National Science Foundation grant DMS-1347804, the Science Center of Williams College, the Williams College Finnerty Fund, and the Clare Boothe Luce Program of the Henry Luce Foundation. We would like to thank Johann Gaebler, Xiaoyu Xu and Zirui Zhou, the   other members of the SMALL 2016 ergodic theory group, for useful discussions and their continuing support. Finally, we would like to thank Isaac Loh for his help in getting the project started.

\section{Preliminaries}

\subsection{Main definitions}

Let $(X, \B, \mu)$ be a $\sigma$-finite, nonatomic Lebesgue measure space. A transformation $T: X \to X$ is said to be \emph{measurable} if for all $A \in \B$, $\TInv(A) \in \B$. An \emph{invertible measurable} transformation is an invertible transformation $T$ such that both $T$ and $T^{-1}$ are measurable. A measurable transformation $T$ is called \emph{measure-preserving} if for all $A \in \B$, we have $\mu(\TInv(A)) = \mu(A)$; it is called \emph{nonsingular} if for all $A \in \B$, $\mu(\TInv(A)) = 0$ if and only if $\mu(A) = 0$. All transformations we consider are assumed to be invertible measurable transformations unless stated otherwise. Moreover, equalities are always assumed to hold up to sets of measure zero.

Let $G$ be a locally compact Polish topological group. A \emph{measurable action} $\T$ of $G$ on $X$ consists of a family of transformations $\{T^g\}_{g \in G}$ such that the map $G \times X \to X$ given by $(g,x) \mapsto T^g(x)$ is Borel measurable, and such that for a.e. $x \in X$ we have $T^g(T^h(x)) = T^{gh}(x)$ and $T^e(x) = x$, where $g,h \in G$ and $e$ is the identity of $G$. An action $\T = \{T^g\}$ is said to be \emph{measure-preserving} if every transformation $T^g$ is measure-preserving, and is said to be \emph{nonsingular} if every transformation $T^g$ is nonsingular. For our purposes, $G$ will be either $\Z^d$ or $\R$. If $G = \R$, the action is called a \emph{flow}, and when $G = \Z$, the action is simply a transformation.

An action $\T$ is called \emph{ergodic} if whenever $A \in \B$ satisfies $T^g(A) = A$ for all $g \in G$, we have $\mu(A) = 0$ or $\mu(X \setminus A) = 0$; $\T$ is called \emph{conservative} if for all sets $A$ of positive measure, there exists $g \in G$ such that $\mu(T^g(A) \cap A) > 0$. One can show that if $\mu$ is $\sigma$-finite and non-atomic, and $\T$ is an ergodic action, then $\T$ is also conservative. In the case of transformations, or $\Z$-actions, ergodicity is equivalent to the condition that for all sets $A$ and $B$ of positive measure, there exists $n > 0$ such that $\mu(T^n A \cap B) > 0$.

Given two $G$-actions $\T = \{T^g\}_{g \in G}$ and $\ES = \{S^g\}_{g \in G}$ on spaces $X$ and $Y$, respectively, we define their \emph{product action} $\T \times \ES = \{(T \times S)^g\}_{g \in G}$ on $X \times Y$ by
\[(T \times S)^g(x,y) := (T^g x, S^g y).\]
We note that if $\T \times \ES$ is ergodic, then both $\T$ and $\ES$ are ergodic.

For a nonsingular transformation $T$, a set $W \in \B$ with $\mu(W) > 0$ is called \emph{wandering} if all the images $T^n(W)$ ($n \in \Z$) are disjoint. Let $\D(T)$ denote the union of all wandering sets, called the \emph{dissipative part} of $T$. We call $\C(T) := X \setminus \D(T)$ the \emph{conservative part} of $T$, and the partition $X = \D \sqcup C$ is called the \emph{Hopf decomposition} of $T$. It is easy to see that $T$ is conservative if and only if it admits no wandering set, or alternatively $X = \C(T)$. If we have $X = \D(T)$ instead, then we say that $T$ is \emph{totally dissipative}. It is well known that if $T$ and $S$ are conservative and ergodic, then either $T \times S$ is conservative or $T \times S$ is totally dissipative (see Proposition 1.2.4 in \cite{Aa97}).

A transformation $T$ is called \emph{partially rigid} if there exists $0 < \alpha \leq 1$ and an increasing sequence $\{n_i\}$ such that for all sets $A$, $0 < \mu(A) < \infty$, we have
\[\liminf_{i \to \infty} \mu(T^{n_i}A \cap A) \ge \alpha \mu(A).\]
If $\alpha = 1$, we say that $T$ is \emph{rigid}. We say that a transformation $T$ has \emph{ergodic index k} if $T \times \dots \times T$ ($k$ times) is ergodic but $T \times \dots \times T$ ($k+1$ times) is not ergodic; $T$ has \emph{infinite ergodic index} if $T \times \dots \times T$ ($k$ times) is ergodic for every $k$.

The notion of density for subsets of $\N$ will arise frequently. Given a set $E \subset \N$, its \emph{upper density} is defined as
\[\text{Dens}(E) := \limsup_{n \to \infty} \frac{| E \cap \{1, \dots, n\} |}{n}.\]
If the limit exists we simply call $\text{Dens}(E)$ the \emph{density} of the set $E$. We remark that finite intersections of sets of density 1 also have density 1 (which is analogous to saying that finite unions of sets of density 0 have density 0).

Given $E, F \subset \R$, we define their \emph{sumset} to be
\[E \oplus F = \{x + y: x \in E, y \in F\},\]
and their \emph{difference set} to be
\[E - F = \{x - y: x \in E, y \in F\}.\]

As mentioned in the introduction, the notion of conservative sequences is of fundamental importance to our methods of proof. Since the more general notion of a conservative set under an action also comes up, the next definition is stated in this context.

\begin{definition} \label{conservative sets}
Let $\T$ be a nonsingular action of a group $G$, and let $A$ be a set of positive measure. We define the \emph{conservative set} of $A$ under $\T$ as
\[C_T(A) := \{g \in G: \mu(T^g(A) \cap A) > 0\}.\]
In the case where $G = \Z$, we call $C_T(A)$ a \emph{conservative sequence}.
\end{definition}

We note that if $g \in C_\T(A)$, then its inverse $-g \in C_\T(A)$ as well. Thus, for the case of $\Z$-actions and $\R$-actions, conservative sets are ``symmetric" about 0, and therefore it is enough to study the positive numbers in $C_\T(A)$. A similar remark holds for conservative sets of $\Z^d$-actions.

\subsection{Rank-one actions}

\subsubsection{Rank-one transformations} \label{conservative sequences}

Many of our constructions will be rank-one cutting-and-stacking transformations defined on $[0,\infty)$. One constructs inductively a sequence of columns. A \emph{column} $C_n$ consists of a finite sequence of $h_n$ disjoint intervals $I_{n,0}, \dots, I_{n,h_n-1}$ of the same length, where we think of $I_{n,i+1}$ as sitting above $I_{n,i}$. The intervals composing $C_n$ are called \emph{levels} and $h_n$ is called the \emph{height} of $C_n$. When the context is clear, we may use $C_n$ to refer to the union of the levels in $C_n$. If $I_{n,k}$ is a level of $C_n$, its \emph{height} is $h(I) = k$. Moreover, if $I$ and $J$ are two levels in the same column, we define the \emph{distance} between $I$ and $J$ to be $d(I, J) := |h(I) - h(J)|$.

We will start with $C_0 = \{[0,1)\}$. Suppose $C_n$ is defined. To obtain $C_{n+1}$ from $C_n$, we cut $C_n$ into $r_n$ subcolumns and add $s_{n,i}$ ($i = 0, \dots, r_n-1$) spacers to the $i$-th subcolumn. We then obtain column $C_{n+1}$ by stacking each subcolumn on top of the subcolumn to its left. We choose each spacer interval so that it is disjoint from all previously chosen spacers and from $[0,1)$, and so that it abuts the previously chosen spacer (or, if it is the first spacer, so that it abuts $[0,1)$).

Each column $C_n$ is associated with a column map $T_{C_n}$ that sends each level in $C_n$ to the level immediately above it via the unique orientation-preserving translation. The rank-one transformation $T$ is then defined as the pointwise limit of the maps $T_{C_n}$. For our purposes, the number of spacers we add will always be so large that the resulting space $X$ will be $[0,\infty)$.

Let $A$ be a level of some column $C_n$ of height $h(A)$ (in $C_n$), and let $D(A,m)$, $m \geq n$, be the set of heights of copies of $A$ in $C_m$ (called the \emph{descendant set} of $A$ in column $C_m$). If we write $h_{n,k} := h_n + s_{n,k}$ (where $k=0,1,\dots,r_n-1$), then it is easy to see that
\[D(A,n+1) = \{h(A)\} \cup \left\{h(A) + \sum_{k=0}^i h_{n,k}: i = 0,\dots,r_n-2\right\}.\]
Define
\[H_{n} = \{0\} \cup \left\{\sum_{k=0}^i h_{n,k}: i=0,\dots,r_n-2\right\},\]
which is just equal to $D(J,n+1)$ where $J$ is the base level of $C_n$.
Then for $m > n$ we obtain
\[D(A,m) = h(A) + H_n \oplus H_{n+1} \oplus \dots \oplus H_{m-1}.\]
We let
\[C_T^m(A) := D(A,m) - D(A,m),\]
which corresponds to the ``return times of $A$" within column $C_m$. We note that $C_T^m(A) \subset C_T^{m+1}(A)$ for all $m$ and $C_T(A) = \bigcup_{m=n}^\infty C_T^m(A)$.

We now consider a special class of infinite rank-one transformations, which we call \emph{skyscraper transformations}. At stage $n$, we cut column $C_n$ into 2 subcolumns and we only add spacers to the right subcolumn 
and also
\[C_T^m(A) = \{c_n h_n + c_{n+1} h_{n+1} + \dots + c_{m-1} h_{m-1}: c_i = -1,0 \text{ or } +1\}.\]
right subcolumn. In this case, if $A$ is a level of $C_n$, then for all $m > n$, we have
\[D(A,m) = h(A) + \{0,h_n\} \oplus \{0,h_{n+1}\} \oplus \dots \oplus \{0,h_{m-1}\}\]

	\begin{figure}[htp!]
		\begin{tikzpicture}[scale=2]
			\draw (0,0) -- (1,0);
				\node at (0,-.25){0};
				\node at (.5,-.25){$C_0$};
				\node at (1,-.25) {1};
			\draw [->] (1.3,.25) -- (1.7,.25);
			\draw (2,0) -- (3,0);
				\draw (2.5,-.05) -- (2.5,.05); 
					\draw [->] (2.75, .05) -- (2.75,.15);
				\draw [dashed] (2.5,.2) -- (3, .2); 
					\draw [->] (2.75, .25) -- (2.75,.35);
				\draw [dashed] (2.5,.4) -- (3, .4);
				\draw [dotted] (2.75, .5) -- (2.75, .7);
				\draw [dashed](2.5, .8) -- (3,.8);
					\draw [->] (2.75, .85) -- (2.75,.95);
				\draw [dashed] (2.5, 1) -- (3, 1);
				
					\node at (2,-.25) {0};
					\node at (3,-.25) {1};
					
					\node at (2.5,-.85) {$C_0\to C_1$};	
				\draw [->] (3.3,.25) -- (3.7,.25);
				\draw (2.25,.25) -- (2.25, .5) -- (1.85,.5) -- (1.85,-.45) -- (2.75,-.45);
				\draw [->] (2.75, -.45) -- (2.75, -.15);
			\draw (4,0) -- (4.5,0);
			\draw (4,.2) -- (4.5,.2);
			\draw [dashed] (4, .4) -- (4.5,.4);
			\draw [dashed] (4,.6) -- (4.5,.6);
			\draw [dotted] (4.25,.7) -- (4.25,.9);
			\draw [dashed] (4, 1) -- (4.5,1);
			\draw [dashed] (4, 1.2) -- (4.5, 1.2);
				\node at (4.25,-.25){$C_1$};
				\node at (4, -.1){\tiny{0}};
				\node at (4.5,-.1){\tiny{1/2}};
				\node at (4,.1){\tiny{1/2}};
				\node at (4.5,.1){\tiny{1}};
		\end{tikzpicture}
		\bigskip
		\caption{Construction of column $C_1$ from column $C_0$ for a skyscraper transformation}
		\end{figure}
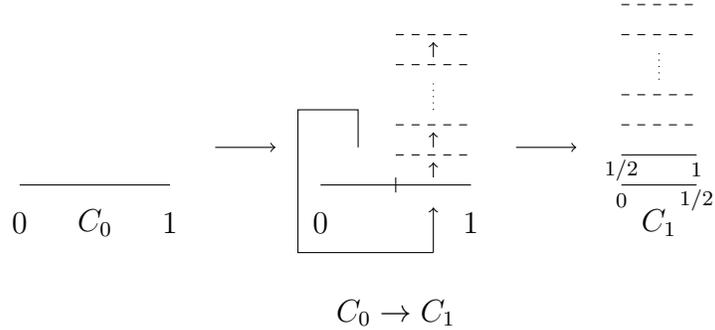

\bigskip	

\subsubsection{Skyscraper $\Z^d$-actions}

We present the analogue of skyscraper transformations for $\Z^d$-actions. The reader should note that when $d = 1$, the following description agrees with the one given in the previous section.

We follow the notation and terminology used in \cite{MRS99}. Given positive integers $\ell_1, \dots, \ell_d$, let $\SQ(\ell_1, \dots, \ell_d) := \{(x_1,\dots, x_d) \in \Z^d: 0 \leq x_i < \ell_i \text{ for each } i=1, \dots, d\}$, and define a \emph{grid} $G$ of dimensions $\ell_1 \times \dots \times \ell_d$ to be a bijection between $\SQ(\ell_1, \dots, \ell_d)$ and a collection of intervals of equal length. We call an interval $I$ in the range of $G$ a \emph{level}. When the context is clear, we may use $G$ to denote the union of the intervals in its range. We define the \emph{location} of $I \in G$ to be $\text{Loc}(I) := G^{-1}(I)$.

Fix positive integers $a_1, \dots, a_d$ and positive integers $h_n$ ($n \in \N$) such that $h_{n+1} \geq 2 h_n$. We describe the construction of a \emph{skyscraper $\Z^d$-action} with parameters $a_1, \dots, a_d$. Let $G_0$ consist of the interval $[0,1)$ (or more precisely $G_0$ is the map which sends the origin in $\Z^d$ to $[0,1))$. Now suppose $G_n$ has been defined and has domain $\SQ(h_{n,0}, \dots, h_{n,d-1})$. To obtain $G_{n+1}$, cut each interval in $G_n$ into $2^d$ subintervals of equal length. We now define $G_{n+1}$, which will be a bijection between $\SQ(a_1 h_{n+1}, \dots, a_d h_{n+1})$ and the collection of subintervals coming from the cuts together with new spacer intervals. Suppose an interval $I$ in $G_n$ has location $(x_1, \dots, x_d)$. Enumerate the subintervals of $I$ coming from the cuts as $I_0$, \dots, $I_{2_d - 1}$, and express each $j \in \{0, \dots, 2^d - 1\}$ by its binary representation
\[j = b_{j,0} + b_{j,1} \cdot 2 + b_{j,2} \cdot 2^2 + \dots + b_{j, d-1} \cdot 2^{d-1},\]
where $b_{j,k} \in \{0,1\}$ for each $k$. Then assign the subinterval $I_j$ to the location $(x_1, \dots, x_d) + (b_{j,0} h_{n,0}, \dots, b_{j, d-1} h_{n,d-1})$. Finally, we assign spacers to the elements of $\SQ(a_1 h_{n+1}, \dots, a_d h_{n+1})$ that have not yet been assigned intervals.

To each grid $G_n$ we associate \emph{grid maps} $T_{G_n}^{(1,\dots,0)}, \dots, T_{G_n}^{(0,\dots, 1)}$. Given an interval $I$ with location $(x_1, x_2, \dots, x_d)$, we define $T_{G_n}^{(1,\dots,0)}$ on $I$ to be the translation that maps $I$ to the interval with location $(x_1+1, x_2, \dots, x_d)$. If no such interval exists, then $T_{G_n}^{(1,\dots,0)}$ remains undefined. We define the other grid maps in an analogous manner. Finally, the basis transformations $T^{(1,\dots,0)}, \dots, T^{(0,\dots,1)}$ are defined as the pointwise limits of the respective grid maps. It is easy to see that all these basis transformations then commute, so they define a valid action of $\Z^d$ on $\bigcup_{n = 0}^\infty G_n$, which we call a \emph{skyscraper $\Z^d$-action} (with parameters $a_1, \dots, a_d$).

\subsubsection{Rank-one flows}

Since the construction of rank-one flows is similar to that of rank-one transformations, we omit certain details in the description that follows. Again, one constructs a sequence of columns $C_n$, where now a column consists of a finite sequence of \emph{blocks} of the form $B_i = [a_i,b_i) \times [c_i,d_i)$ having the same width $b_i-a_i$ (where $a_i, b_i, c_i, d_i \in \R$). The \emph{height} of $C_n$ is defined as
\[h_n := \sum_i (d_i - c_i).\]
Again, when the context is clear, we may use $C_n$ to refer to the union of the blocks composing it. Further one can think of a column as occupying the space of a rectangle $R_n = [0,\alpha_n) \times [0, \beta_n)$. One then obtains a natural bijection between the column $C_n$ (thought of as a union of blocks) and the rectangle $R_n$ which we can denote $\Loc$. Given $(x,y) \in C_n$, we call $\Loc(x,y)$ the \emph{location} of $(x,y)$. We use $\Loc_1(x,y)$ to denote the $x$-coordinate of $\Loc (x,y)$ and $\Loc_2(x,y)$ to denote the $y$-coordinate of $\Loc (x,y)$.

Let $C_0 = [0,1) \times [0,1)$. If $C_n$ has been defined, we cut $C_n$ into $r_n$ subcolumns and possibly add spacer blocks $S_{n,i}$ ($i = 0, \dots, r_n-1$) to the subcolumns. Then $C_{n+1}$ is obtained by stacking each subcolumn on top of the subcolumn to its left. If the measure of the spacer blocks is large enough, then the union of the columns $C_n$ will be a subset $X$ of $[0,\infty) \times [0, \infty)$ of infinite measure (so we may as well choose $X = [0, \infty) \times [0, \infty)$).

Each column $C_n$ is associated to a so-called column flow $\T = \{T_{C_n}^s\}_{s \in \R}$ defined by
\[T_{C_n}^s(x,y) := \Loc^{-1}(\Loc_1(x,y), \Loc_2(x,y) + s),\]
where $(x,y) \in C_n$. Of course, for any $s \not = 0$, $T_{C_n}^s$ is not defined on all of $C_n$. Our rank-one flow is obtained by letting $T^s$ be the pointwise limit of the column maps $T_{C_n}^s$.

If we always cut into 2, and only add spacer blocks to the right subcolumn, then we call the resulting flow a \emph{skyscraper flow}.

\section{Non-conservativity of product actions}

\subsection{$\Z^d$-actions}

\begin{definition}
Let $\T = \{T^\bfv\}_{\bfv \in \Z^d}$ be a nonsingular action of $\Z^d$ on $X$, and let $a_1, \dots, a_d > 0$. Let
\[E = \{(c_1 a_1, \dots, c_d a_d): c_i = -1, 0 \text{ or } 1 \text{ for each } i\}.\]
If for all $\bfv = (v_1, \dots, v_d) \in E$, the vector $\bfu_\bfv := \frac{1}{\text{gcd}(v_1, \dots, v_d)} \bfv$ is such that $T^{\bfu_\bfv}$ is ergodic, then we say that $\T$ is \emph{squarely ergodic w.r.t. $a_1, \dots, a_d$}. We say that $\T$ is \emph{squarely ergodic} if there exist $a_1, \dots, a_d > 0$ such that $\T$ is squarely ergodic w.r.t. $a_1, \dots, a_d$.
\end{definition}

In this section we prove the following theorem. As mentioned in the introduction, Schmidt and Walters \cite{ScWa82} prove a more general result in the sense that $\T$ is assumed to be a properly ergodic action of a locally compact second countable abelian group without a finite invariant measure, but they do not obtain that $\S$ is rank-one.  

\begin{theorem} \label{Zd actions}
Let $\T$ be an infinite measure-preserving action of $\Z^d$ on $X$ such that $\T$ is squarely ergodic. Then there exists a rank-one skyscraper $\Z^d$-action $\ES$ such that $\T \times \ES$ is not conservative.
\end{theorem}

We first show that if $\T = \{T^\bfv\}_{\bfv \in \Z^d}$ admits a set $A$ of positive measure whose conservative set $C_\T(A)$ has certain gaps, then one can inductively construct a skyscraper $\Z^d$-action $\ES$ such that if $I$ is a level of the first grid $G_1$, then $C_\T(A) \cap C_\ES(I) = \{(0, \dots, 0)\}$. In other words, $C_{\T \times \ES}(A \times I) = \{(0, \dots, 0)\}$, and thus $\T \times \ES$ is not conservative. After that, we show that every $\Z^d$-action $\T$ that satisfies the conditions of Theorem \ref{Zd actions} admits such a set $A$.

\begin{definition}
Let $C \subset \Z^d$, and let $a_1, \dots, a_d > 0$. We say that $C$ has \emph{adequate gaps} w.r.t. $a_1, \dots, a_d$ if for every $n \in \mathbb{N}$, there exists $\ell_n \in \mathbb{N}$ such that
\[[\{(c_1 a_1 \ell_n, \dots, c_d a_d \ell_n): c_i \in \{-1,0,1\} \text{ for all } i\} \oplus [-n,n]^d] \cap C = \varnothing.\]
If $d = 1$, we may use the expression \emph{long gaps} instead. In this case, $C$ has long gaps if and only if $C$ is not a syndetic set.
\end{definition}

As for the case of transformations (see the end of section \ref{conservative sequences}), if $I$ is a level of grid $G_n$, then we can define the sets $C_\ES^m(I)$ as the set of times at which $I$ returns to itself within grid $G_m$ (where $m \geq n$). More precisely, if we let $D(I, m)$ (called the \emph{descendant set} of $I$ in column $C_m$) denote the set of locations of copies of $I$ in $G_m$, then we define $C_\T^m(I) := D(I,m) - D(I,m)$. Once again, we have $C_\ES^m(I) \subset C_\ES^{m+1}(I)$ for all $m \geq n$, and $C_\ES(I) = \bigcup_{m = n}^\infty C_\ES^m(I)$. 
The skyscraper $\Z^d$-actions $\ES$ mentioned in the next lemma will have parameters $a_1, \dots, a_d$. That is, each grid $G_n$ will be a $d$-dimensional rectangle with dimensions $a_1 h_n \times \dots \times a_d h_n$ for some $h_n \in \N$.

For these transformations, we have
\[ C_\ES^m(I) = C_1 \times \dots \times C_d, \]
where $C_i$ denotes the conservative sequence of $I$ under the basis transformation $T^{(0, \dots, 1, \dots, 0)}$, where the $1$ is at the $i$-th position. In other words,
\[C_i = \{c_1 a_i h_1 + \dots + c_m a_i h_m: c_i \in \{-1,0,1\} \text{ for all } i\}.\]
Hence, we can write
\begin{equation} \label{Zd conservative set}
C_\ES^{m+1}(I) = C_\ES^m(I) \oplus \{(c_1 a_1 h_m, \dots, c_d a_d h_m): c_i \in \{-1,0,1\} \}.
\end{equation}

\begin{lemma} \label{Construction of S}
Let $\T$ be a $\Z^d$-action on $X$ that admits a set $A$ of positive measure whose conservative set $C_\T(A)$ has adequate gaps w.r.t. $a_1, \dots, a_d$. Then there exists a skyscraper $\Z^d$-action $\ES$ with parameters $a_1, \dots, a_d$ such that $\T \times \ES$ is not conservative.
\end{lemma}
\begin{proof}
We construct the action $\ES$ so that if $I$ is the level of $G_1$ located at $(0, \dots, 0)$, then $C_{\T \times \ES}(A \times I) = \{(0, \dots, 0)\}$, which is equivalent to $C_\T(A) \cap C_\ES(I) = \{(0, \dots, 0)\}$. As mentioned, each grid of $\ES$ will be a $d$-dimensional rectangle with dimensions $a_1 h_n \times \dots \times a_d h_n$, so to define $\ES$ it is enough to inductively specify each $h_n$. We have $G_0 = \{[0,1)\}$ so $h_0 = 1$. Suppose $h_{n-1}$ has been defined so that $C_\ES^n(I) \cap C_\T(A) = \{(0, \dots, 0)\}$. Since
\[C_\ES^{n+1}(I) = C_\ES^n(I) \oplus \{(c_1 a_1 h_n, \dots, c_d a_d h_n): c_i \in \{-1,0,1\} \}\]
and $C_\T(A)$ has adequate gaps w.r.t. $a_1, \dots, a_d$, we can choose $h_n$ to ensure that $C_\ES^{n+1}(I) \cap C_\T(A) = \{(0, \dots, 0)\}$. This completes the proof since
\[C_S(I) = \bigcup_{m=1}^\infty C_S^m(I).\]
\end{proof}

\bigskip

\begin{figure}[htp!]
		\begin{tikzpicture} 
			\draw [->] (-3.25,0) -- (3.25,0);
			\draw [->] (0,-3.25) -- (0,3.25);
			\draw [ultra thick](-.5,-.5) rectangle (.5,.5);
			\draw [ultra thick](-3,-.5) rectangle (-2,.5);
			\draw [ultra thick](2,-.5) rectangle (3,.5);
			\draw [ultra thick](-.5,2) rectangle (.5,3);
			\draw [ultra thick](-.5,-3) rectangle (.5,-2);-.75,
			\draw [ultra thick](-3,2) rectangle (-2,3);
			\draw [ultra thick](2,2) rectangle (3,3);
			\draw [ultra thick](-3,-3) rectangle (-2,-2);
			\draw [ultra thick](2,-2) rectangle (3,-3);
				\draw [<->, thick] (-.75,0.1) -- (-.75,2.5);
					\node at (-1,1.5) {$h_n$};
				\draw [<->, thick] (.1,-.75) -- (2.5,-.75);
					\node at (1.5,-1) {$h_n$};
			\draw [ultra thick] (5,0) rectangle (6,1);
				\node at (5.5,-.8) {$G_n$};
			\draw [->, thick] (6.25,0) -- (6.75,0);
			\draw [<->, thick] (7,0) -- (7,4);
				\node at (6.65,2) {$h_n$};
			\draw [ultra thick, fill=gray] (7.2,0) rectangle (11.15,4);
				\node at (9,2.5) {SPACERS};
			\draw [ultra thick, fill=white] (7.2,0) rectangle (8.2,1);
				\node at (7.7,.5) {\tiny{$G_n^{(0,0)}$}};
			\draw [ultra thick, fill=white] (8.2,1) rectangle (9.2,2);
				\node at (8.7,1.5) {\tiny{$G_n^{(1,1)}$}};
			\draw [ultra thick, fill=white] (7.2,1) rectangle (8.2,2);
				\node at (7.7,1.5) {\tiny{$G_n^{(0,1)}$}};
			\draw [ultra thick, fill=white] (8.2,0) rectangle (9.2,1);
				\node at (8.7,.5) {\tiny{$G_n^{(0,1)}$}};
			\draw [<->, thick] (7.2,-.25) -- (11.15,-.25);
				\node at (9.5,-.6) {$h_n$};
		\end{tikzpicture}
		\bigskip
				\caption{The image on the left represents adequate gaps w.r.t. $a_1 = a_2 = 1$ for a $\Z^2$-action $\T$. In the above proof, we choose the numbers $h_n$ for $\ES$ so that if $C_S^n(I)$ is contained in the central square, then $C_S^{n+1}(I)$ will be contained in the union of the nine squares. The image on the right illustrates the construction of grid $G_{n+1}$ from grid $G_n$.}
		\end{figure}
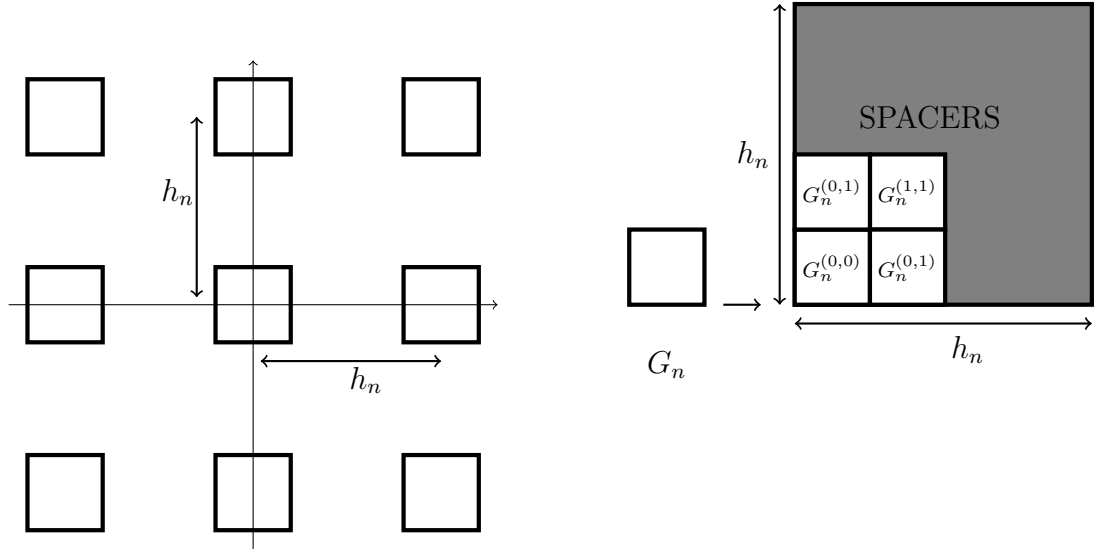

\bigskip

Lemma~\ref{adequate gaps} for the case of integer actions was proved by Hajian \cite{Ha64}
when giving a characterization of transformations not admitting a finite invariant measure, see also \cite{EHIP14}. Our proof generalizes their methods to obtain an analogous result for squarely ergodic $\Z^d$ actions. We let $E := \{(c_1 a_1, \dots, c_d a_d): c_i \in \{-1,0,1\}\}$. The main tool we use is the following corollary of the ergodic theorem:
\begin{proposition} \label{BET}
Let $T$ be an infinite measure-preserving ergodic transformation, and let $A, B \subset X$ of finite measure. Then
\[\lim_{n \to \infty} \frac{1}{n} \sum_{k=0}^{n-1} \mu(T^k A \cap B) = 0.\]
\end{proposition}
Since $\{\mu(T^k A \cap B)\}_{k \in \N}$ is bounded, it follows that for the sets $A, B$ above, $\mu(T^n A \cap B) \to 0$ in density 1 (see \cite{Si08} for example). Recall that if $\bfv = (v_1, \dots, v_d) \in \Z^d$, then we let $\bfu_\bfv := \frac{1}{\text{gcd}(v_1, \dots, v_d)} \bfv$.


\begin{lemma}[Adequate Gaps] \label{adequate gaps}
Let $\T$ be an infinite measure-preserving $\Z^d$-action that is squarely ergodic w.r.t. $a_1, \dots, a_d$, i.e. each transformation $T^{\bfu_\bfv}$, $\bfv \in E$, is ergodic. Let $A^* \subset X$ with $0 < \mu(A^*) < \infty$, and fix $0 < \ve < \mu(A^*)$. Then there exists a set $A \subset A^*$, $\mu(A) > \mu(A^*) - \ve$, such that $A$ has adequate gaps w.r.t. $a_1, \dots, a_d$.
\end{lemma}
\begin{proof}
Let $A_0 = A^*$ and $\ell_0 = 0$, and fix $0 < \ve < \mu(A^*)$. We construct
\[A_0 \supset A_1 \supset A_2 \supset \dots\]
\[\ell_0 < \ell_1 < \ell_2 < \dots\]
such that for every $n \in \N$ and for every $\bfv \in E$, we have
\[ T^{\ell_n \bfv + (k_1, \dots, k_d)}(A_n) \cap A_n = \varnothing\]
for all $k_1, \dots, k_d \in \{-n, \dots, n\}$. Then we will take
\[A = \bigcap_{n=1}^\infty A_n.\]

Suppose $A_n$ and $\ell_n$ are defined. Let
\[B_n := \bigcup_{k_1, \dots, k_d \in \{-n-1, \dots, n+1\}} T^{(k_1, \dots, k_d)}(A_n).\]
By the corollary of the infinite ergodic theorem, for each $\bfu_\bfv$, $\bfv \in E$, we have
\begin{equation} \label{crucial step}
\lim_{m \to \infty} \mu(T^{m \bfu_\bfv}(B_n) \cap A_n) = 0 \text{ in density 1}.
\end{equation}
It is not hard to see that this implies that
\begin{equation} \label{crucial step 2}
\lim_{m \to \infty} \mu(T^{m \bfv}(B_n) \cap A_n) = 0 \text{ in density 1}.
\end{equation}
Since the intersection of finitely many sets of density 1 still has density 1, we conclude that there exists $\ell_{n+1} > \ell_n$ such that
\[\mu( \bigcup_{\bfv \in E} T^{\ell_{n+1} \bfv}(B_n) \cap A_n) < \frac{\ve}{2^{n+1}}.\]
Now define
\[A_{n+1} := A_n \setminus \bigcup_{\bfv \in E} T^{\ell_{n+1} \bfv}(B_n).\]
For each $\bfv \in E$ and for each $k_1, \dots, k_d \in \{-n-1, \dots, n+1\}$, we have
\[T^{\ell_{n+1} \bfv + (k_1, \dots, k_d)}(A_{n+1}) \cap A_{n+1} \subset T^{\ell_{n+1} \bfv + (k_1, \dots, k_d)}(A_{n}) \setminus \bigcup_{\bfv \in E} T^{\ell_{n+1} \bfv}(B_n) = \varnothing.\]
Moreover, the fact that $\mu(A_{n+1}) > \mu(A_n) - \frac{\ve}{2^{n+1}}$ for each $n$ ensures that
\[\mu(A) = \mu \left(\bigcap_{n=0}^\infty A_n \right) > \mu(A^*) - \ve.\]
\end{proof}

\begin{remark}
Since the proof of the Adequate Gaps Lemma remains valid for non-invertible transformations, Theorem \ref{Zd actions} is also true in the non-invertible case. The same goes for certain actions of semigroups of $\Z^d$.
\end{remark}

\subsection{Extensions to $d=1$, i.e. transformations}

In this context we can actually prove the following more general result.

\begin{theorem}
Let $T$ be a nonsingular transformation with no finite invariant measure. Then there exists a rank-one transformation $S$ such that $T \times S$ is not conservative. Moreover, $S$ can be chosen to be rigid or have infinite ergodic index.
\end{theorem}

For clarity, we will prove the following two results in turn.
\begin{enumerate}
\item There exists a skyscraper transformation $S$ such that $T \times S$ is not conservative.
\item The transformation $S$ can be chosen to be rigid or have infinite ergodic index.
\end{enumerate}

\begin{proof}[Proof of (1)]
It suffices to show that there exists a set $A$ of positive measure for $T$ such that $C_T(A)$ has long gaps. To guarantee this, we use the following characterization of a transformation with no finite invariant measure for which the reader can refer to  \cite{EHIP14}:
\[\text{There exists a set $A$, $\mu(A) > 0$, such that } \frac{1}{n} \sum_{k=0}^{n-1} \mu(T^k A) \to 0.\]
Again, this implies that for this set $A$, $\mu(T^k A) \to 0$ in density 1 (since, without loss of generality, we can assume that $\mu$ is a nonsingular probability measure). We can then proceed as in the proof of the Adequate Gaps Lemma to obtain the desired gaps.
\end{proof}

For the proof of (2), we use the following easy lemma, whose proof is left to the reader.

\begin{lemma} \label{density lemma}
Suppose $A \subset \N$ has density 1, and suppose $B = \{b_m\}_{m \in \N} \subset \N$ has positive density, i.e.
\[\lim_{n \to \infty} \frac{|B \cap \{1, \dots, n\}|}{n}\]
exists and is positive. Then the set of $m \in \N$ such that $b_m \in A$ forms a set of density 1.
\end{lemma}

\begin{proof}[Proof of (2)]
We work with the assumption that $T$ is an infinite measure-preserving ergodic transformation. The same arguments work for nonsingular transformations with no finite invariant measure with only minor modifications.

We first show that $S$ can be chosen to be a rigid rank-one transformation. Let $C_0 = [0,1)$. Suppose $C_n$ is defined. To obtain $C_{n+1}$ cut $C_n$ into $r_n$ subcolumns and only add spacers on top of the right subcolumn. We impose that $\limsup_{n \to \infty} r_n = \infty$. The resulting rank-one transformation $S$ can be finite or infinite depending on the number of spacers added at each stage. It is easy to see, though, that $S$ is rigid along the sequence $n_i = h_i$ of heights. For simplicity, we will take $r_n = n$ below.

Let $I$ denote the base level of column $C_1$. Then for $m \geq 1$ we have
\[C_S^{m+1}(I) = C_S^m(I) \oplus \{c_{m} h_{m}: c_{m} \in \{-r_{m}+1,\dots, r_{m}-1\}\}.\]
We wish to inductively construct $S$ (by defining the heights $h_i$) such that for some set $A$, $C_T(A) \cap C_S^n(I) = \{0\}$ for all $n$. The idea is to find the right notion of gaps for $C_T(A)$ and prove a suitable version of the Adequate Gaps Lemma. More precisely, we require the following gaps: for every $n > 0$, there exists $\ell_{n} > 0$ such that
\[[k \ell_{n} - n, k \ell_{n} + n] \cap C_T(A) = \varnothing \text{ for all } k = 1, 2, \dots, n.\]
Now if we have defined $h_1, \dots, h_{n-1}$ in such a way that $C_T(A) \cap C_S^n(I) = \{0\}$, then we can choose $h_n$ (amongst the $\ell_n$) such that
\[C_T(A) \cap C_S^{n+1}(I) = C_T(A) \cap [C_S^n(I) \oplus \{(-r_n + 1) h_n, \dots, (r_n -1) h_n\}] = \varnothing\]
using the existence of the gaps.

We give a sketch of the appropriate version of the Adequate Gaps Lemma, using the same setup and notation. We have \[B_n := \bigcup_{j \in \{0, 1, \dots, n+1\}} T^j(A_n).\] By the corollary of the ergodic theorem,
\[\lim_{m \to \infty} \mu(T^m B_n \cap A_n) = 0 \text{ in density 1}.\]
But by Lemma \ref{density lemma}, and the fact that $k \Z$ has positive density, we even have that for all $k \in \N$,
\[\lim_{m \to \infty} \mu(T^{km}B_n \cap A_n) = 0 \text{ in density 1}.\]
Using that the intersection of finitely many sets of density 1 has density 1, we can therefore find $\ell_{n+1} > \ell_n$ such that $\mu(T^{k \ell_n} B_n \cap A_n)$ is arbitrarily small for all $k = 1, \dots, r$. The remainder of the proof is the same as for the Adequate Gaps Lemma

\bigskip

Similarly, $S$ can be chosen to be a rank-one transformation with infinite ergodic index. We refer the reader to \cite{AdFrSi01} for a more detailed discussion of the construction. Let $C_0 = \{[0,1)\}$. If $C_n$ is defined, we obtain $C_{n+1}$ by cutting $C_n$ into 4 subcolumns and adding a certain number of spacers to each subcolumn. We require the following for each $n \in \N$:
\begin{itemize}
\item $h_{n,0} = h_{n,2}+1$ 
\item $h_{n,1} > n \cdot (h_{n,0} + h_{n,2} + 2 h_n - 2 s_{n-1,3})$
\item $s_{n,3}$ must be sufficiently large, i.e. $h_{n,3} > n \cdot (h_{n,0} + h_{n,1} + h_{n,2} + h_n)$ and
\[\lim_{n \to \infty} \frac{h_n}{h_n - s_{n-1,3}} = \infty.\]
\end{itemize}
To simplify our analysis, we will replace the second condition with $h_{n,1} = 5 n h_{n,0}$. Let $I$ denote the base level of $C_1$ for $S$. Again, the idea is to find the right notion of gaps for $C_T(A)$ and prove a suitable version of the Adequate Gaps Lemma. Note that for all $n \geq 1$, we have
\[C_S^{n+1}(I) = C_S^{n}(I) \oplus \left\{ c \sum_{k=i}^j h_{n,k}: 0 \leq i \leq j \leq 2, c \in \{-1, 0, 1\} \right\}.\]
It is not difficult to see that we require the following gaps: for every $m>0$ there exists $\ell_m>0$ such that
\[[\ell_m - m, \ell_m + m] \cap C_T(A) = \varnothing;\]
\[[5m \ell_m - m, 5m \ell_m + m] \cap C_T(A) = \varnothing;\]
\[[6 m \ell_m - m, 6 m \ell_m + m] \cap C_T(A) = \varnothing;\]
\[[(6m +1) \ell_m - m, (6m + 1) \ell_m + m] \cap C_T(A) = \varnothing.\]
Then we can use the existence of these gaps to inductively choose $h_{n,0}, h_{n,1}, h_{n,2}$ to be some $\ell_m$, $5 m \ell_m$ and $\ell_m + 1$, respectively, so that $C_T(A) \cap C_S^n(I) = \{0\}$ (and we also need $h_{n,3}$ to be sufficiently large). We leave the details to the reader.
\end{proof}

\begin{theorem}
Let $\{T_i\}_{i \in \N}$ be a countable collection of nonsingular transformations with no finite invariant measure. Then there exists a rank-one transformation $S$ such that $T_i \times S$ is not conservative for every $i \in \N$.
\end{theorem}
\begin{proof}
Using a modification of the Adequate Gaps Lemma, we can prove that there exist sets $A_i$ of positive measure (for each $T_i$) such that for any $n>0$, the conservative sequences $C_{T_1}(A_1), \dots, C_{T_n}(A_n)$ share arbitrarily long gaps (i.e. for any $k>0$, there exists $\ell_k > 0$ such that $C_{T_i}(A_i) \cap [\ell_k - k, \ell_k + k] = \varnothing$ for all $i = 1, \dots, n$). To show this, we can proceed by induction and guarantee that at the $n$-th stage the sets $C_{T_1}(A_1), \dots, C_{T_n}(A_n)$ all share a gap of length $n$.

We now give the inductive construction of the skyscraper transformation $S$ by specifying the heights $h_i$. Let $J_i$ denote the base level of $C_i$. Suppose $h_1, \dots, h_{n-1}$ have been defined so that for each $i = 1, \dots, n-1$, we have
\[C_{T_i}(A_i) \cap C_S^n(J_i) = \varnothing.\]
Using that for any level $J$ in a column $C_i$ ($i=1, \dots, n$), we have
\[C_S^{n+1}(J) = C_S^n(J) \oplus \{-h_n, 0, h_n\},\]
and the existence of the gaps described above, we can choose $h_n$ from $\ell_k$ so that for each $i = 1, \dots, n$, we have
\[C_{T_i}(A_i) \cap C_S^{n+1}(J_i) = \varnothing.\]
It then follows that for the $S$ we construct, the conservativity condition fails for $T_i \times S$ for the rectangle $A_i \times J_i$ (for each $i$).
\end{proof}

\begin{remark}
The transformation $S$ can be chosen to be a rank-one transformation that is rigid or has infinite ergodic index.
\end{remark}

\begin{remark}
Every time one of the products $T \times S$ is not conservative, we know it must be totally dissipative.
\end{remark}

\subsection{Flows}

Observe that measure-preserving flows $\T = \{T^r\}_{r \in \R}$ are always conservative under our definition. Indeed, by the continuity of flows, given a set $A$ of positive measure, there exists $\delta > 0$ such that if $0 < r < \delta$, then $\mu(T^r A \cap A) > 0$. Thus to extend our result to flows, we need to introduce the following notion. We say that a measure-preserving flow (or transformation) is \emph{strongly conservative} if for all sets $A$ of positive measure, the conservative set $C_\T(A)$ is unbounded. We note that in the case of transformations, conservativity is equivalent to strong conservativity. On the other hand, there exist flows that are conservative but not strongly conservative (e.g. consider a shift on the real line $T^s(x) = x + s$). It is also clear that for measure-preserving flows, ergodicity implies strong conservativity.

We prove the following theorem, which will be an immediate consequence of Lemmas \ref{Flow lemma} and \ref{Isaac's lemma for flows}.

\begin{theorem} \label{flows}
Let $\T = \{T^r\}_{r \in \R}$ be an infinite measure-preserving ergodic flow. Then there exists a rank-one skyscraper flow $\ES = \{S^r\}_{r \in \R}$ such that the product $\T \times \ES$ is not strongly conservative (and therefore not ergodic).
\end{theorem}

Suppose $B$ denotes the lowest block of column $C_1$ of a skyscraper flow $\ES$ (so $B = [0,\frac{1}{2}) \times [0,1)$). As for skyscraper transformations, we can derive a useful expression for $C_\ES^m(B)$, the return times of $B$ within column $C_m$:
\[C_\ES^m(B) = [0,1) \oplus \{c_1 h_1 + \dots + c_{m-1} h_{m-1}: c_i \in \{-1,0,1\} \text{ for all } i\}.\]
We say that a set $E \subset \R$ has \emph{long gaps} if for every $n > 0$, there exists $\ell_n > 0$ such that $E \cap [\ell_n - n, \ell_n + n] = \varnothing$. Then the following lemma is clear based on our previous work for transformations:

\begin{lemma} \label{Flow lemma}
Let $\T = \{T^r\}_{r \in \R}$ be a flow that admits a set $A$ of positive measure whose conservative set $C_\T(A)$ has long gaps. Then there exists a skyscraper flow $\ES$ such that $\T \times \ES$ is not strongly conservative.
\end{lemma}
\begin{proof}
We construct a skyscraper flow $\ES$ so that if $B$ is the lowest block of $C_1$, then $C_\T(A) \cap C_\ES(B) \subset (-1,1)$. The procedure is so similar to that in Lemma \ref{Construction of S} that we leave it to the reader.
\end{proof}

Finally, we show that every infinite measure-preserving ergodic flow admits a set $A$ of positive measure with long gaps in $C_\T(A)$. To establish the Adequate Gaps Lemma for transformations, the key step was deriving (\ref{crucial step}):
\[\lim_{m \to \infty} \mu(T^m B_n \cap A_n) = 0.\]
To do so, we used:
\begin{enumerate}
\item the   ergodic theorem in the form of Proposition \ref{BET}.
\item the fact that finite unions of sets of the form $T^k(A)$ have finite measure provided that $A$ has finite measure.
\end{enumerate}

Now for flows we have the following infinite ergodic theorem:

\begin{proposition} [\cite{Aa97}] \label{BET for flows}
Let $\T = \{T^r\}_{r \in \R}$ be an infinite measure-preserving ergodic flow. Then for all sets $A$ and $B$ of finite measure,
\[\lim_{T \to \infty} \frac{1}{T} \int_0^T \mu(T^r A \cap B) \; dr = 0.\]
In particular,
\[\liminf_{r \to \infty} \mu(T^r A \cap B) = 0.\]
\end{proposition}

It is slightly more difficult to obtain an analogue of (2). We would like to guarantee that a union of the form $\bigcup_{r \in [0,\alpha)} T^r (A)$ has finite measure if $A$ has finite measure. In general, though, such a union need not even be measurable, and even if it is measurable, it is not clear that it must be of finite measure.

We will need the following:

\begin{proposition} \label{A-K theorem}
Let $\T = \{T^r\}_{r \in \R}$ be an ergodic measure-preserving flow. Then it is isomorphic to a flow built under a function. In particular, there exists $A^* \subset X$, $0 < \mu(A^*) < \infty$, such that
\[\mu \left(\bigcup_{r \in [0, \alpha)} T^r(A^*) \right) < \infty\]
for some $\alpha > 0$ (and therefore for \emph{any} $\alpha > 0$ using the measure-preserving property).
\end{proposition}

The first part is a version of the Ambrose-Kakutani representation theorem which appears in \cite{Sa84} as Theorem A, and the second part is an easy corollary. We refer the reader to \cite{Sa84} for the definition of flows \emph{built under a function} and further discussion.

\begin{lemma} \label{Isaac's lemma for flows}
Let $\T = \{T^r\}_{r \in \R}$ be an infinite measure-preserving ergodic flow. Then there exists a set $A$ of positive measure whose conservative set $C_\T(A)$ has long gaps.
\end{lemma}
\begin{proof}
Choose $A^*$ so that $\mu \left(\bigcup_{r \in [0, \alpha)} T^r(A^*) \right) < \infty$ for every $\alpha > 0$. Let $A_0 = A^*$ and $\ell_0 = 0$. Fix $0 < \ve < \mu(A^*)$. We construct
\[A_0 \supset A_1 \supset A_2 \supset \dots\]
\[\ell_0 < \ell_1 < \ell_2 < \dots\]
such that for every $n \in \N$ and for all $r \in [0, n)$, we have
\[ T^{\ell_n + r}(A_n) \cap A_n = \varnothing.\]
Then we will take $A = \bigcap_{n=1}^\infty A_n$.
Suppose $A_n$ and $\ell_n$ are defined. Let $B_n := \bigcup_{r \in [0, n+1)} T^r(A_n)$. By assumption $B_n$ has finite measure. Thus the infinite   ergodic theorem for flows yields the existence of $\ell_{n+1} > \ell_n$ such that
\[\mu(T^{\ell_{n+1}}(B_n) \cap A_n) < \frac{\ve}{2^{n+1}}.\]
Now define $A_{n+1} := A_n \setminus T^{\ell_{n+1}}(B_n)$.
The remainder of the proof is similar enough to the proof for $\Z^d$-actions that we leave it to the reader.
\end{proof}


\section{Conservative non-ergodic products}
\label{S:cnerg}

Say that a transformation $T$ satisfies the {\it strict conservative multiplier property} if there exists
an infinite conservative ergodic measure-preserving transformation $S$ such that $T\times S$ is 
conservative but not ergodic. When that if $T$ is an infinite measure-preserving  K-automorphism (see \cite{Pa65} for the definition) and $S$ is ergodic, then whenever $T \times S$ is conservative, the product must also be ergodic (see \cite[Proposition 4.8b]{SiTh95}). Thus $K$-automorphisms 
do not satisfy the strict conservative multiplier property. In Theorem~\ref{Vadim's theorem} we give a condition on rank-one transformations so that it satisfies the strict conservative multiplier property; Corollary~ref{C:Vcor} gives a simpler to verify condition for this property. It would be interesting to know
what the largest class of transformations is that satisfy this property.

\begin{theorem} \label{Vadim's theorem}
Let $T$ be a rank-one transformation with bounded cuts, and satisfying
\begin{equation} \label{Vadim's condition}
h_n < 2 (h_1 + s_{1,r_1-1} + \dots + s_{n-1, r_{n-1}-1}) + \min_{0 \leq j \leq r_n-2} s_{n,j} - 1
\end{equation}
for all $n \geq 2$. Then there exists a rank-one skyscraper transformation $S$ such that $T \times S$ is conservative but not ergodic.
\end{theorem}

\begin{corollary}\label{C:Vcor}
If $T$ is an (infinite) rank-one transformation such that $s_{n-1, r_{n-1}-1} \geq \frac{h_n}{2}$ for all $n \in \N$, then there exists a skyscraper transformation $S$ such that $T \times S$ is conservative but not ergodic.
\end{corollary}

\begin{remark}
The corollary covers many important examples of infinite rank-one transformations such as the infinite Chac\'on transformation and the Hajian-Kakutani skyscraper transformation.
\end{remark}

Before proceeding to the proof, we make a couple of observations.
\begin{enumerate}
\item If $T$ is a rank-one transformation and $I$ is the base level of column $C_1$, then for all $m \geq 2$,
\begin{multline} \label{ConsTM}
C_T^m(I) = C_T^{m-1}(I) \\
\oplus \{c \sum_{k = i}^j h_{m,k}: c \in \{-1,0,1\}, 0 \leq i \leq j < r_m-1 \}.
\end{multline}
Thus,
\[\begin{aligned} 
\max C_T^m(I) &= (h_2 - h_{1,r_1-1}) + \dots + (h_m - h_{m-1,r_{m-1}-1}) \\
&= (h_2 - h_1 - s_{1,r_1-1}) + \dots + (h_m - h_{m-1} - s_{m-1,r_{m-1}-1}) \\
&= h_m - h_1 - (s_{1,r_1-1} + \dots + s_{m-1,r_{m-1}-1}).
\end{aligned}\]
\item Rank-one transformations with bounded cuts, i.e. $\sup r_n < \infty$, are partially rigid with $\alpha = \frac{1}{\sup r_n}$ and sequence $n_i$ consisting of infinitely many terms of the form $\sum_{k=i}^j h_{n,k}$, where $n \in \N$ and $0 \leq i \leq j \leq r_n-2$. We note that if $T$ and $S$ are both partially rigid along sequences that intersect in an infinite subsequence, then $T \times S$ is also partially rigid and thus conservative.
\end{enumerate}

\begin{proof}[Proof of Theorem \ref{Vadim's theorem}]
Let $T$ be as in the statement of the theorem, and let $S$ be the skyscraper transformation where the heights $H_i$ are given by
\[H_i := h_{n_i+1} - h_{n_i, r_{n_i}-1},\]
where $n_i$ is an increasing sequence in $\N$ chosen so that $H_i \geq 2 H_{i-1}$ for all $i$. Then since both $T$ and $S$ are partially rigid along the sequence $n_i = H_i$, their product $T \times S$ is conservative.

Let $J$ be the base level of column $C_1$ for $S$. Then if $m \geq 1$,
\[C_S^{m+1}(J) = C_S^m(J) \oplus \{-H_i, 0, H_i\}.\]
Letting $I$ be the base level of column $C_1$ for $T$, we show that
\[[C_S(J) + 1] \cap C_T(I) = \varnothing.\]
This will give us that $(T \times S)^n (I \times J) \cap (I \times S(J)) = \varnothing$ for all $n \in \Z$, and therefore that $T \times S$ is not ergodic.

\begin{figure}[htp!]
		\begin{tikzpicture}[scale=0.5, transform shape]
			\draw [->] (0,0) -- (10,0);
				\draw (.25,-.5) rectangle (.75,.5);
					\node at (.5,-1) {\tiny{C$^n_S(J)-H_n$}};
				\draw (4.75,-.5) rectangle (5.25,.5);
					\node at (5,-1) {\tiny{C$_S^n(J)$}};
				\draw (9.25,-.5) rectangle (9.75,.5);
					\node at (9.5,-1) {\tiny{C$_S^n(J)+H_n$}};
			\draw [->] (0,-3) -- (10,-3);
				\draw (.25,-3.5) rectangle (.75,-2.5);
					\node at (.5,-4) {\tiny{C$_T^n(I)-H_n$}};
				\draw (1.25,-3.5) rectangle (1.75,-2.5);
					\draw [->] (1.5,-4.3) -- (1.5,-3.9);
					\node at (1.5,-4.5) {\tiny{C$_T^n(I)-H_n+min(h_{n,0},h_{n,r_n-2})$}};
			\draw (5,-.1) -- (5,.1);
			\node at (5,-.25) {\tiny{0}};
				
				\draw [dotted] (2.15,-2.8) -- (3.35,-2.8);
				\draw (3.75,-3.5) rectangle (4.25,-2.5);
				\draw (4.75,-3.5) rectangle (5.25,-2.5);	
					\node at (5,-4) {\tiny{C$^n_T(I)$}};
					\draw (5,-3.1) -- (5,-2.9);
						\node at (5,-3.25) {\tiny{0}};
				\draw (5.75,-3.5) rectangle (6.25,-2.5);
				\draw [dotted] (6.65,-2.8) -- (7.85,-2.8);
				\draw (8.25,-3.5) rectangle (8.75,-2.5);
					\draw [->] (8.5,-4.3) -- (8.5,-3.9);
					\node at (8.5,-4.5) {\tiny{C$^n_T(I)+H_n-min(h_{n,0},h_{n,r_n-2})$}};
				\draw (9.25,-3.5) rectangle (9.75,-2.5);
					\node at (9.5,-4) {\tiny{C$_T^n(I)+H_n$}};
					
		\end{tikzpicture}
		\bigskip
		\caption{The top and bottom images represent collections of intervals, the union of which contain $C_S^{n+1}(J)$ and $C_T^{n+1}(I)$, respectively.}
		\end{figure}
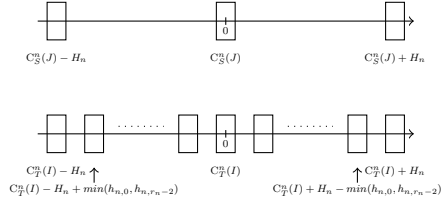

\bigskip

\noindent \textbf{Claim:} If $[C_S(J) + 1] \cap C_T(I) \not = \varnothing$, then there must exist $n \in \N$ such that
\[[C_S^n(J) + 1] \cap C_T^n(I) \not = \varnothing.\]
\begin{proof}[Proof of claim]
Note that one can rewrite (\ref{Vadim's condition}) as
\begin{multline}
h_n - h_1 - (s_{1,r_1-1} + \dots s_{n-1, r_{n-1}-1}) + 1 \\
< [\min_{0 \leq j \leq r_n-2} h_{n,j}] - (h_n - h_1 - (s_{1,r_1-1} + \dots + s_{n-1, r_{n-1}-1}),
\end{multline}
which just says that the maximum of $C_S^n(J) + 1$ (which is equal to the maximum of $C_T^n(J) + 1$) is strictly less than the minimum positive element of $C_T^{n+1}(I)$ that is not in $C_T^n(I)$. Since the sets $C_S^n(A)$ are always symmetric about $0$, this is enough to conclude that
\begin{equation} \label{disjoint}
[C_S^n(J) + 1] \cap [C_T^{n+1}(I) \setminus C_T^{n}(I)] = \varnothing.
\end{equation}
for all $n \in \N$. By repeating this argument, one can even conclude that for all $m > n$,
\begin{equation}
[C_S^n(J) + 1] \cap [C_T^{m}(I) \setminus C_T^{n}(I)] = \varnothing.
\end{equation}
\end{proof}

\bigskip

We now prove by induction that for all $n \in \N$,
\begin{equation} \label{hypothesis}
[C_S^n(J) + 1] \cap C_T^n(I) = \varnothing,
\end{equation}
which would complete the proof based on the Claim. The base case is clear since $C_S^1(J) + 1 = \{1\}$ and $C_T^1(I) = \{0\}$. Assume that (\ref{hypothesis}) holds for some $n \in \N$. Recall that
\[C_S^{n+1}(J) = C_S^n(J) \oplus \{-H_n,0,H_n\}.\]
We prove each of the following:
\begin{enumerate}
\item $[C_S^n(J) + 1] \cap C_T^{n+1}(I) = \varnothing$
\item $[C_S^n(J) - H_n + 1] \cap C_T^{n+1}(I) = \varnothing$
\item $[C_S^n(J) + H_n + 1] \cap C_T^{n+1}(I) = \varnothing$
\end{enumerate}

\noindent
\underline{Case 1}: By (\ref{disjoint}) and the inductive hypothesis, $[C_S^n(J) + 1] \cap C_T^{n+1}(I) = \varnothing$.

\bigskip

\noindent
\underline{Case 2}: By Case 1, we have
\begin{equation} \label{Ind hyp}
[C_S^n(J) - H_n + 1] \cap [C_T^{n}(I) - H_n] = \varnothing.
\end{equation}
Thus, to prove case 2, it suffices to show
\begin{equation} \label{Case2}
\max C_S^n(J) - H_n + 1 < \min C_T^n(I) - [H_n - \min(h_{n,0}, h_{n,r_n-2})].
\end{equation}
This follows from (\ref{ConsTM}) and the observation that $H_n - \min(h_{n,0}, h_{n,r_n-2})$ is the second largest element of the form $\sum_{k=i}^j h_{m,k}$.
Now (\ref{Case2}) can be rewritten as
\[H_1 + \dots + H_{n-1} - H_n + 1 < -(H_1 + \dots + H_{n-1}) - [H_n - \min(h_{n,0}, h_{n,r_n-2})].\]
But this in turn is equivalent to
\[2 (h_n - h_1 - [s_{1,r_1-1} + s_{2,r_2-1} + \dots + s_{n-1,r_{n-1}}]) + 1 < h_n + \min(s_{n,0}, s_{n,r_n-2}),\]
which follows from (\ref{Vadim's condition}).

\bigskip

\noindent
\underline{Case 3}: This follows from the inductive hypothesis and the fact that
\[\max C_T^n(I) + [H_n - \min(h_{n,0}, h_{n,r_n-2})] \leq \min C_S^n(J) + H_n,\]
which can be proven by repeating the argument used in Case 2.
\end{proof}

\begin{corollary} \label{Vadim's corollary}
Let $T$ be a rank-one transformation with bounded cuts, and satisfying for all $n \geq 1$,
\begin{equation} \label{inequality}
h_{n+1} - h_n < 2 s_{n, r_n-1} + \min_{0 \leq j \leq r_{n+1}-2} s_{n+1,j} - \min_{0 \leq j \leq r_n-2} s_{n, j}.
\end{equation}
\end{corollary}
\begin{proof}
We prove by induction that (\ref{Vadim's condition}) holds for every $n \geq 2$. For the base case, note that
\[h_1 < 2 h_1 + \min_{0 \leq j \leq r_1-2} s_{1, j} -1.\]
Then by adding the inequality (\ref{inequality}) for $n = 1$ to both sides, we obtain
\[h_2 < 2 (h_1 + s_{1,r_1-1}) + \min_{0 \leq j \leq r_2-2} s_{2, j} - 1.\]
For the inductive step, suppose that (\ref{Vadim's condition}) holds for some $n \geq 2$. Then by adding (\ref{inequality}) to both sides again, we are done.
\end{proof}


\section{Ergodic products of rank-one transformations}

We recall that in \cite{AaLiWe79} it is shown that for every conservative ergodic infinite measure-preserving transformation there exists a conservative ergodic Markov shift $S$ so that $T\times S$ 
is ergodic. In this section we show that we can choose $S$ to be rank-one, but $T$ 
has to be restricted to be a rank-one transformation with bounded cuts. 

In the previous section, we observed that rank-one transformations with bounded cuts are partially rigid. A slight generalization of this observation, which we use below, is that if $I$ and $J$ are levels in the same column, then there exist infinitely many $n>0$ such that $\mu(T^n(I) \cap J) \geq \alpha \mu(I)$. In the proof below, we use $\mu$ to denote Lebesgue measure on the real line.

\begin{theorem}
Suppose $T$ is a rank-one transformation with bounded cuts defined on $X =[0,1)$ or $[0,\infty)$. Then there exists a rank-one skyscraper transformation $S$ such that $T \times S$ is ergodic.
\end{theorem}
\begin{proof}
We can inductively find an increasing sequence $n_i$ such that $n_i > 2 n_{i-1}$ for each $i$, and such that for each ordered pair of levels $(I_m, J_m)$ coming from the same column (where we consider pairs of levels from \textit{all} columns) there exists an infinite subsequence $\{n_{i_k}\}$ such that for each $k$,
\[\mu(T^{n_{i_k}}I_m \cap J_m) \geq \alpha \mu(I_m) = \alpha \mu(J_m).\]
Now define $S$ to be the skyscraper transformation with heights $h_i = n_i$ (we therefore need to require $n_i \geq 2 n_{i-1}$ so that $S$ is well-defined). Take $A, B \subset X \times Y$ with $\mu \times \mu(A) > 0$ and $\mu \times \mu(B) > 0$. We show that there exists $k>0$ such that $\mu \times \mu((T \times S)^k(A) \cap B) > 0$.

Fix $0 < \ve < \frac{\alpha}{4}$. We can choose levels $I$ and $L$ in column $C_m^T$ and levels $J$ and $M$ in column $C_\ell^S$ such that
\[\mu \times \mu (A \cap (I \times J)) >  (1-\ve) \mu(I) \mu(J)\]
\[\mu \times \mu (B \cap (L \times M)) >  (1-\ve) \mu(L) \mu(M).\]
Without loss of generality, we can assume that $M$ is above $J$ in column $C_\ell^S$. It is also not difficult to see that one can find a copy $L^*$ of $L$ in $C_{m+j}^T$ (for some $j$) such that
\[\mu \times \mu (B \cap (L^* \times M)) > (1-\ve) \mu(L^*) \mu(M)\]
while having that $L^*$ is at least $d(J, M)$ above the bottom level of $C_{m+j}^T$.

One can find a copy $I_1$ of $I$ in $C_{m+j}^T$ such that $I_1 \times J$ is still $(1-\ve)$-full  of $A$. Let $I^*$ denote the level in $C_{m+j}^T$ that is a distance of $d(J,M)$ below $L^*$. By construction, $n_i$ is a partial rigidity sequence for $S$ so there exists some $n_i$ such that
\[\mu(T^{n_i}I_1 \cap I^*) \geq \alpha \mu(I^*)\]
\[\mu(S^{n_i}J \cap J) \geq \frac{1}{2} \mu(J).\]
Hence, for this $n_i$, we have
\[\mu \times \mu ((T \times S)^{n_i}(I_1 \times J) \cap (I^* \times J)) \geq \frac{\alpha}{2} \mu(I^*) \mu(J),\]
and hence
\[\begin{aligned}
\mu \times \mu ((T \times S)^{n_i}A \cap (I^* \times J)) & \geq \left(\frac{\alpha}{2} - \ve \right) \mu(I^*) \mu(J) \\
& > \frac{\alpha}{4} \mu(I^*) \mu(J).
\end{aligned}\]
We then obtain
\[\mu \times \mu ((T \times S)^{n_i + d(J,M)} A \cap B) > \left( \frac{\alpha}{4} - \ve \right) \mu(I^*) \mu(J) > 0.\]
\end{proof}

\begin{remark}
For completeness, we remark that the transformation $S$ above can be chosen to be rigid or have infinite ergodic index.
\end{remark}

\newcommand{\M}{\mathcal{M}}
\newcommand{\E}{\mathcal{E}}

\section{Category results concerning product transformations}

Let $(X, \B, \mu)$ be an infinite Lebesgue space (for simplicity, we will think of $X$ as $\R$ and $\mu$ as Lebesgue measure). We write $\M(\mu)$ for the set of all invertible transformations on $X$ that preserve $\mu$, and write $\E(\mu) \subset \M(\mu)$ for the set of all ergodic transformations. One can define a topology on $\M(\mu)$ called the \emph{weak topology}, which we briefly review. A \emph{dyadic interval} refers to an interval of the form $[\frac{k}{2^j}, \frac{k+1}{2^j})$, where $k \in \Z$ and $j \in \N$, and by a \emph{dyadic set} we mean a finite union of dyadic intervals. We note that the collection of dyadic sets $\{D_i\}_{i \in \N}$ forms a dense algebra of $(\R, \B, \mu)$, and define the distance between two transformations $T$ and $S$ in $\M(\mu)$ to be
\[d(T, S) := \sum_{i = 1}^\infty \frac{\mu(T^{-1}(D_i) \triangle S^{-1}(D_i))}{2^i \mu(D_i)}.\]
It can be shown that $d$ is a complete, separable metric, and the topology one obtains is called the \emph{weak topology} on $\M(\mu)$. 
In \cite{Sa71}, Sachdeva shows that $\E(T)$ is a dense $G_\delta$ set in $\M(T)$ (i.e. ergodicity is a \emph{generic} property).

\begin{definition} \label{three classes}
Given $T \in \E(\mu)$, we define the following subsets of $\E(\mu)$:
\begin{enumerate}
\item $E(T) := \{S \in \E(\mu): T \times S \text{ is ergodic}\}$
\item $\bar{E}C(T) := \{S \in \E(\mu): T \times S \text{ is conservative but not ergodic}\}$
\item $\bar{C}(T) := \{S \in \E(\mu): T \times S \text{ is not conservative}\}$
\end{enumerate}
\end{definition}

For a given transformation $T \in \M(\mu)$, we study the Baire categories of the sets $E(T)$, $\bar{E}C(T)$ and $\bar{C}(T)$. We first introduce a few notions. A transformation $T \in \M(\mu)$ is called \emph{anti-periodic} if for every $n \in \N$ and for all sets $A \subset X$, there exists $B \subset A$ such that $T^{-n}(B) \setminus B \not = \varnothing$ and $B \setminus T^{-n}(B) \not = \varnothing$. Given a transformation $T \in \M(\mu)$ and $\rho \in \R$, the transformations $T \cdot \rho$ and $\rho \cdot T$ are defined by $(T \cdot \rho)(x) = T(\rho \cdot x)$ and $(\rho \cdot T)(x) = \rho \cdot (Tx)$. For us, the \emph{conjugacy class} of $T$ will be the set
\[\mathscr{C}(T) := \{\rho^{-1} \cdot R^{-1} \circ T \circ R \cdot \rho: R \in \M(\mu), \rho \in (0, \infty)\}.\]

\begin{theorem}
Let $T \in \M(\mu)$, and let $S$ be an anti-periodic transformation. Then:
\begin{enumerate}
\item If $S \in E(T)$, then $E(T)$ is a dense $G_\delta$ set in $\M(\mu)$.
\item If $S \in \bar{E}C(T)$, then $\bar{E}C(T)$ is dense in $\M(\mu)$.
\item If $S \in \bar{C}(T)$, then $\bar{C}(T)$ is dense in $\M(\mu)$.
\end{enumerate}
\end{theorem}
\begin{proof}
It is easy to check that if $S \in E(T)$, $S \in \bar{E}C(T)$ or $S \in \bar{C}(T)$, then we respectively have $\mathscr{C}(S) \subset E(T)$, $\mathscr{C}(S) \subset \bar{E}C(T)$ or $\mathscr{C}(S) \subset \bar{C}(T)$. In \cite{Sa71}, Sachdeva shows that if $S$ is an anti-periodic transformation, then its conjugacy class $\mathscr{C}(S)$ is dense in $\M(\mu)$.

It remains to show that the existence of an $S$ in $E(T)$ implies that $E(T)$ is a $G_\delta$ set. Consider the space $\M(\mu \times \mu)$ of invertible measure-preserving transformations on the product space $(X \times X, \B \otimes \B, \mu \times \mu)$, and let $A := \{T\} \times \M(\mu)$. Note that one can identify the set $A$ (considered with the subspace topology) with the space $\M(\mu)$. Since the intersection of a $G_\delta$ set with a subspace is also a $G_\delta$ set (in that subspace), the set $A \cap \E(\mu \times \mu)$ is a $G_\delta$ set. But $A \cap E(X \times X)$ exactly corresponds to $E(T)$ under the identification, so $E(T)$ is itself a $G_\delta$ set.
\end{proof}

One can check that rank-one transformations are anti-periodic. Thus our results have the following interesting corollaries:
\begin{enumerate}
\item For any rank-one transformation $T$ with bounded cuts, $E(T)$ is a dense $G_\delta$ set in $\M(\mu)$.
\item For any rank-one transformation $T$ satisfying condition (\ref{Vadim's condition}), $\bar{E}C(T)$ is dense in $\M(\mu)$.
\item For any $T \in \E(\mu)$, $\bar{C}(T)$ is dense in $\M(\mu)$.
\end{enumerate}

Further, it  is that the property of having ergodic cartesian square $T \times T$ is a generic property \cite{Sa71}. The reader may refer to the survey \cite{DaSi09} for other properties. Since rank-one transformations also form a generic set (see \cite{BSSSW15}), we conclude that for a generic transformation $T \in \M(\mu)$, the set $E(T)$ is a dense $G_\delta$ set.

\normalsize
\medskip

\bibliographystyle{plain}
\bibliography{Paper_Oct_4}

\begin{thebibliography}{10}

\bibitem{Aa97}
J.~Aaronson.
\newblock {\em An introduction to infinite ergodic theory}, volume~50 of {\em
  Mathematical Surveys and Monographs}.
\newblock American Mathematical Society, Providence, RI, 1997.

\bibitem{AaLiWe79}
Jonathan Aaronson, Michael Lin, and Benjamin Weiss.
\newblock Mixing properties of {M}arkov operators and ergodic transformations,
  and ergodicity of {C}artesian products.
\newblock {\em Israel J. Math.}, 33(3-4):198--224 (1980), 1979.
\newblock A collection of invited papers on ergodic theory.

\bibitem{AdFrSi01}
Terrence Adams, Nathaniel Friedman, and Cesar~E. Silva.
\newblock Rank-one power weakly mixing non-singular transformations.
\newblock {\em Ergodic Theory Dynam. Systems}, 21(5):1321--1332, 2001.

\bibitem{BSSSW15}
Francisc Bozgan, Anthony Sanchez, Cesar~E. Silva, David Stevens, and Jane Wang.
\newblock Subsequence bounded rational ergodicity of rank-one transformations.
\newblock {\em Dyn. Syst.}, 30(1):70--84, 2015.

\bibitem{Br70}
A.~Brunel.
\newblock New conditions for existence of invariant measures in ergodic theory.
\newblock In {\em Contributions to {E}rgodic {T}heory and {P}robability
  ({P}roc. {C}onf., {O}hio {S}tate {U}niv., {C}olumbus, {O}hio, 1970)}, pages
  7--17. Springer, Berlin, 1970.

\bibitem{DaSi09}
Alexandre~I. Danilenko and Cesar~E. Silva.
\newblock Ergodic theory: non-singular transformations.
\newblock In {\em Mathematics of complexity and dynamical systems. {V}ols.
  1--3}, pages 329--356. Springer, New York, 2012.

\bibitem{EHIP14}
Stanley Eigen, Arshag Hajian, Yuji Ito, and Vidhu Prasad.
\newblock {\em Weakly wandering sequences in ergodic theory}.
\newblock Springer Monographs in Mathematics. Springer, Tokyo, 2014.

\bibitem{FuWe78}
Hillel Furstenberg and Benjamin Weiss.
\newblock The finite multipliers of infinite ergodic transformations.
\newblock In {\em The structure of attractors in dynamical systems ({P}roc.
  {C}onf., {N}orth {D}akota {S}tate {U}niv., {F}argo, {N}.{D}., 1977)}, volume
  668 of {\em Lecture Notes in Math.}, pages 127--132. Springer, Berlin, 1978.

\bibitem{Ha64}
Arshag Hajian.
\newblock Strongly recurrent transformations.
\newblock {\em Pacific J. Math.}, 14:517--523, 1964.

\bibitem{MRS99}
E.~J. Muehlegger, A.~S. Raich, C.~E. Silva, M.~P. Touloumtzis, B.~Narasimhan,
  and W.~Zhao.
\newblock Infinite ergodic index {${\bf Z}^d$}-actions in infinite measure.
\newblock {\em Colloq. Math.}, 82(2):167--190, 1999.

\bibitem{Pa65}
William Parry.
\newblock Ergodic and spectral analysis of certain infinite measure preserving
  transformations.
\newblock {\em Proc. Amer. Math. Soc.}, 16:960--966, 1965.

\bibitem{Sa71}
Usha Sachdeva.
\newblock On category of mixing in infinite measure spaces.
\newblock {\em Math. Systems Theory}, 5:319--330, 1971.

\bibitem{Sa84}
Ryotaro Sato.
\newblock On the ratio maximal function for an ergodic flow.
\newblock {\em Studia Math.}, 80(2):129--139, 1984.

\bibitem{ScWa82}
Klaus Schmidt and Peter Walters.
\newblock Mildly mixing actions of locally compact groups.
\newblock {\em Proc. London Math. Soc. (3)}, 45(3):506--518, 1982.

\bibitem{Si08}
C.~E. Silva.
\newblock {\em Invitation to ergodic theory}, volume~42 of {\em Student
  Mathematical Library}.
\newblock American Mathematical Society, Providence, RI, 2008.

\bibitem{SiTh95}
Cesar~E. Silva and Philippe Thieullen.
\newblock A skew product entropy for nonsingular transformations.
\newblock {\em J. London Math. Soc. (2)}, 52(3):497--516, 1995.

\end{thebibliography}

\end{document}